\documentclass[11pt]{amsart}

\usepackage{epigamath}

\usepackage[english]{babel}

\numberwithin{equation}{section}

\usepackage[all]{xy}
\usepackage{tikz-cd}
\usepackage{mathrsfs}
\usepackage[text={6.7in,9.2in},centering]{geometry}

\newtheorem{prop} {Proposition} [section]
 
\newtheorem{lem}[prop] {Lemma}

\newtheorem{prop-def}[prop]{Proposition-Definition}

\newtheorem*{thmA}{Theorem A} 
 
\newtheorem*{thmB}{Theorem B}

\newtheorem*{conj}{Conjecture}

\theoremstyle{definition}
\newtheorem{defi}[prop] {Definition}
\newtheorem*{defiint}{Definition} 
\theoremstyle{remark}
\newtheorem{rmk}[prop]{Remark}

\newtheorem*{ackn}{Acknowledgment}

\newcommand{\C}{\mathbb{C}}

\newcommand{\cH}{\mathcal{H}}

\renewcommand{\a}{\alpha}

\newcommand{\f}{\varphi}

\newcommand{\om}{\omega}
\newcommand{\p}{\psi}

\newcommand{\Om}{\Omega}

\newcommand{\ie}{{\it i.e.\ }}

\newcommand{\dbar}{\overline{\partial}}

\DeclareMathOperator{\Aut}{Aut}

\DeclareMathOperator{\reg}{reg}

\EpigaVolumeYear{5}{2021} \EpigaArticleNr{5} \ReceivedOn{January 24,
2020}
\InFinalFormOn{December 21, 2020}
\AcceptedOn{January 21, 2021}

\title{A note on Lang's conjecture for quotients of bounded domains}
\titlemark{Lang's conjecture for quotients of bounded domains}

\author{S\'ebastien Boucksom}
\address{CNRS and CMLS, \'Ecole Polytechnique, 91128 Palaiseau Cedex, France}
\email{sebastien.boucksom@polytechnique.edu}
\author{Simone Diverio}
\address{Dipartimento di Matematica \lq\lq Guido Castelnuovo\rq\rq{}, SAPIENZA Universit\`a di Roma, Piazzale Aldo Moro 5, I-00185 Roma}
\email{diverio@mat.uniroma1.it} 

\authormark{S. Boucksom \and S. Diverio}

\AbstractInEnglish{It was conjectured by Lang that a complex projective manifold is Kobayashi hyperbolic if and only if it is of general type together with all of its subvarieties. We verify this conjecture for projective manifolds whose universal cover carries a bounded, strictly plurisubharmonic function. This includes in particular compact free quotients of bounded domains.}



\TitleInFrench{Une note sur la conjecture de Lang pour les quotients des domaines born\'es}

\AbstractInFrench{Lang a \'emis la conjecture qu'une vari\'et\'e projective complexe est Kobayashi hyperbolique si et seulement si elle est de type g\'en\'eral ainsi que toutes ses sous-vari\'et\'es. Nous v\'erifions cette conjecture pour les vari\'et\'es projectives dont le rev\^etement universel admet une fonction strictement plurisousharmonique born\'ee. Ceci englobe en particulier les quotients compacts des domaines born\'es par un groupe agissant librement.}

\acknowledgement{Both authors are partially supported by the ANR Programme: D\'efi de tous les savoirs (DS10) 2015, \lq\lq GRACK\rq\rq, Project ID: ANR-15-CE40-0003ANR. The second named author is also partially supported by the ANR Programme: D\'efi de tous les savoirs (DS10) 2016, \lq\lq FOLIAGE\rq\rq, Project ID: ANR-16-CE40-0008}

\dedication{\lq\lq{}... Voglio una vita che non \`e mai tardi, di quelle che non dormono mai...\rq\rq{}}

\begin{document}


\removeabove{0.4cm}
\removebetween{0.4cm}
\removebelow{0.4cm}

\maketitle

\begin{prelims}

\DisplayAbstractInEnglish

\bigskip

\DisplayKeyWords

\medskip

\DisplayMSCclass

\bigskip

\languagesection{Fran\c{c}ais}

\bigskip

\DisplayTitleInFrench

\medskip

\DisplayAbstractInFrench

\end{prelims}


\newpage

\setcounter{tocdepth}{1}

\tableofcontents


%
%
\section{Introduction}
%
%

For a compact complex space $X$, Kobayashi hyperbolicity is equivalent to the fact that every holomorphic map $\C\to X$ is constant, thanks to a classical result of Brody. When $X$ is moreover projective (or, more generally, compact K\"ahler), hyperbolicity is further expected to be completely characterized by (algebraic) positivity properties of $X$ and of its subvarieties. More precisely, we have the following conjecture, due to S.~Lang.

\begin{conj}\cite[Conjecture 5.6]{Lan86}
A projective variety $X$ is hyperbolic if and only if every subvariety $($including $X$ itself\,$)$ is of general type.
\end{conj}

Recall that a projective variety $X$ is of general type if the canonical bundle of any smooth projective birational model of $X$ is big, \textsl{i.e.}~has maximal Kodaira dimension. This is for instance the case when $X$ is smooth and \emph{canonically polarized}, \textsl{i.e.} with an ample canonical bundle $K_X$. 

Note that Lang's conjecture in fact implies that every smooth hyperbolic projective manifold $X$ is canonically polarized, as conjectured in 1970 by S.~Kobayashi. It is indeed a well-known consequence of the Minimal Model Program that any projective manifold of general type without rational curves is canonically polarized (see for instance~\cite[Theorem A]{BBP}). 

Besides the trivial case of curves and partial results for surfaces~\cite{MM83,DES79,GG80,McQ98}, Lang's conjecture is still almost completely open in higher dimension as of this writing. General projective hypersurfaces of high degree in projective space form a remarkable exception: they are known to be hyperbolic~\cite{Bro17} (see also~\cite{McQ99,DEG00,DT10,Siu04,Siu15,RY18}), and they satisfy Lang's conjecture~\cite{Cle86,Ein88,Xu94,Voi96,Pac04}. 

\medskip

It is natural to test Lang's conjecture for the following two basic classes of manifolds, known to be hyperbolic since the very beginning of the theory: 
\begin{itemize}
\item[(N)] compact K\"ahler manifolds $X$ with negative holomorphic sectional curvature; 
\item[(B)] compact, free quotients $X$ of bounded domains $\Omega\Subset\C^n$. 
\end{itemize}
In case (N), ampleness of $K_X$ was established in~\cite{WY16a,WY16b,TY17} (see also~\cite{DT16}). By curvature monotonicity, this implies that every smooth subvariety of $X$ also has ample canonical bundle. More generally, Guenancia recently showed~\cite{Gue18} that each (possibly singular) subvariety of $X$ is of general type, thereby verifying Lang's conjecture in that case. One might even more generally consider the case where $X$ carries an arbitrary Hermitian metric of negative holomorphic sectional curvature, which seems to be still open. 

\medskip

In this note, we confirm Lang's conjecture in case (B). While the case of quotients of bounded \emph{symmetric} domains has been widely studied (see, just to cite a few,~\cite{Nad89,BKT13,Bru16,Cad16,Rou16,RT18}), the general case seems to have somehow passed unnoticed. Instead of bounded domains, we consider more generally the following class of manifolds, which comprises relatively compact domains in Stein manifolds, and has the virtue of being stable under passing to an \'etale cover or a submanifold. 

\begin{defiint} We say that a complex manifold $M$ is \emph{of bounded type} if it carries a bounded, strictly plurisubharmonic\footnote{\ie each point of $M$ admits a coordinate neighbourhood on which $\f-c|z|^2$ is psh for some $c>0$.} function $\f$. 
\end{defiint}

By a well-known result of Richberg, any \emph{continuous} bounded strictly psh function on a complex manifold $M$ can be written as a decreasing limit of smooth strictly psh functions, but this fails in general for discontinuous functions~\cite[p.~66]{For}, and it is thus unclear to us whether every manifold of bounded type should carry also a \emph{smooth} bounded strictly psh function.

\begin{thmA}\label{thm:main}
Let $X$ be a compact K\"ahler manifold admitting an \'etale $($Galois$)$ cover $\tilde X\to X$ with $\tilde X$ of bounded type. Then:
\begin{itemize}
\item[(i)] $X$ is Kobayashi hyperbolic;
\item[(ii)] $X$ has large fundamental group;
\item[(iii)] $X$ is projective and canonically polarized;
\item[(iv)] every subvariety of $X$ is of general type.
\end{itemize}
\end{thmA}
Note that $\tilde X$ can always be replaced with the universal cover of $X$, and hence can be assumed to be Galois. Manifolds to which the above theorem applies include compact free quotients of bounded symmetric domains, and Kodaira surfaces (see \S\ref{sect:examples} for more details).

\medskip

By~\cite[Theorem~3.2.8]{Kob98}, (i) in Theorem A holds iff $\tilde X$ is hyperbolic, which follows from the fact that manifolds of bounded type are Kobayashi hyperbolic~\cite[Theorem~3]{Sib81} (alternatively, any entire curve $f:\C\to X$ lifts to $\tilde X$, and the pull-back to $\C$ of the bounded, strictly psh function carried by $\tilde X$ has to be constant, showing that $f$ itself is constant). 

\smallskip

By definition, (ii) means that the image in $\pi_1(X)$ of the fundamental group of any subvariety $Z\subseteq X$ is infinite~\cite[\S 4.1]{Kol}, and is a direct consequence of the fact that manifolds of bounded type do not contain nontrivial compact subvarieties. 
According to the Shafarevich conjecture, $\tilde X$ should in fact be Stein; in case $\tilde X$ is a bounded domain of $\C^n$, this is indeed a classical result of Siegel~\cite{Sie50} (see also \cite[Theorem~6.2]{Kob59}). 

\smallskip

By another classical result, this time due to Kodaira~\cite{Kod}, any compact complex manifold $X$ admitting a Galois \'etale cover $\tilde X\to X$ biholomorphic to a bounded domain in $\C^n$ is projective, with $K_X$ ample. Indeed, the Bergman metric of $\tilde X$ is non-degenerate, and it descends to a positively curved metric on $K_X$. Our proof of (iii) and (iv) is a simple variant of this idea, inspired by~\cite{CZ02}. For each subvariety $Y\subseteq X$ with desingularization $Z\to Y$ and induced Galois \'etale cover $\tilde Z\to Z$, we use basic H\"ormander--Andreotti--Vesentini--Demailly $L^2$-estimates for $\dbar$ to show that the Bergman metric of $\tilde Z$ is generically non-degenerate. It then descends to a psh metric on $K_Z$, smooth and strictly psh on a nonempty Zariski open set, which is enough to conclude that $K_Z$ is big, by~\cite{Bou02}.

The above reasoning actually yields the following variant of Theorem A, which is perhaps worth mentioning. 

\begin{thmB}\label{thm:theoremb}
Let $X$ be a compact K\"ahler manifold whose universal cover carries a bounded psh function which is strictly psh at one point. Then:  
\begin{itemize}
\item[(i)] the manifold $X$ is of general type (and hence projective), 
\end{itemize}
and there exists a closed proper subset $\Sigma\subset X$ (in the Euclidean topology) such that
\begin{itemize}
\item[(ii)] any subvariety of $X$ not contained in $\Sigma$ is of general type; 
\item[(iii)] the manifold $X$ is Kobayashi hyperbolic modulo $\Sigma$; in particular, every entire curve $\mathbb C\to X$ is contained in $\Sigma$.
\end{itemize}
\end{thmB}

A result very similar to (i) of Theorem B appears in~\cite[Proposition 2.3]{Che03}, where the psh function is further assumed to be smooth. See also~\cite{Kik11,Kik13}, especially~\cite[Lemma 3.4]{Kik13}, for other related results. 

Unfortunately, at the moment we are not able to show that the proper subset $\Sigma$ in the statement above can be chosen to be Zariski closed. If it were the case, this would confirm the (strong) Green--Griffiths--Lang conjecture for this class of projective manifolds of general type.

\medskip

As a final comment, note that \emph{K\"ahler hyperbolic manifolds} in the sense of Gromov, \textsl{i.e.} compact K\"ahler manifolds $X$ carrying a K\"ahler metric $\om$ whose pull-back to the universal cover $\pi:\tilde X\to X$ satisfies $\pi^*\om=d\a$ with $\a$ bounded, also satisfy (i)--(iii) in Theorem A \cite{Gro}. It would be interesting to check Lang's conjecture for such manifolds as well.

\begin{ackn} This work was started during the first-named author's stay at SAPIENZA Universit\`a di Roma. He is very grateful to the mathematics department for its hospitality, and to INdAM for financial support. 

The second-named author warmly thanks Laura Geatti for indicating the reference \cite{Han57}.

Both authors would also like to thank Stefano Trapani for helpful discussions, in particular for pointing out the reference~\cite{For}, as well as the anonymous referee for several useful observations and remarks, in particular for suggesting to add the statement which corresponds to Theorem B.
\end{ackn}

%
%
\section{The Bergman metric and manifolds of general type}
%
%

We start by recalling that the \emph{Bergman space} of a complex manifold $M$ is the separable Hibert space $\cH=\cH(M)$ of holomorphic forms $\eta\in H^0(M,K_M)$ such that 
$$
\|\eta\|_\cH^2:=i^{n^2}\int_{\tilde X}\eta\wedge\bar\eta<\infty,
$$
with $n=\dim M$. Assuming $\cH\ne\{0\}$, we get an induced (possibly singular) psh metric $h_M$ on $K_M$, invariant under $\Aut(M)$, characterized pointwise by
$$
h/h_M=\sup_{\eta\in\cH\setminus\{0\}}\frac{|\eta|^2_h}{\|\eta\|_\cH^2}=\sum_j |\eta_j|^2_h,
$$
for any choice of smooth metric $h$ on $K_M$ and orthonormal basis $(\eta_j)$ for $\cH$ (see for instance~\cite[\S 4.10]{Kob98}). 

The curvature current of $h_M$ is classically called the \lq\lq Bergman metric\rq\rq{} of $M$; it is a \textsl{bona fide} K\"ahler form precisely on the Zariski open subset of $M$ consisting of points at which $\cH$ generates $1$-jets~\cite[Proposition 4.10.11]{Kob98}. 

\begin{defi} We shall say that a complex manifold $M$ has a \emph{non-degenerate (resp.~generically non-degenerate) Bergman metric} if its Bergman space $\cH$ generates $1$-jets at each (resp.~some) point of $M$. 
\end{defi}

We next recall the following standard consequence of $L^2$-estimates for $\dbar$. 

\begin{lem}\label{lem:1jet} Let $M$ be a complete K\"ahler manifold with a bounded psh function $\f$. If $\f$ is strictly psh on $M$ (resp.~at some point of $M$), then the Bergman metric of $M$ is non-degenerate (resp.~generically non-degenerate). 
\end{lem}

\begin{proof} Pick a complete K\"ahler metric $\omega$ on $M$. Assume $\f$ strictly psh at $p\in M$, and fix a coordinate ball $(U,z)$ centered at $p$ with $\f$ strictly psh near $\overline U$. Pick also $\chi\in C^\infty_c(U)$ with $\chi\equiv 1$ near $p$. Since $\chi\log|z|$ is strictly psh in an open neighbourhood $V$ of $p$, smooth on $U\setminus\overline V$, and compactly supported in $U$, we can then choose $A\gg 1$ such that
$$
\psi:=(n+1)\chi\log|z|+A\f
$$
is psh on $M$, with $dd^c\p\ge\om$ on $U$. Note that $\psi$ is also bounded above on $M$, $\f$ being assumed to be bounded. 

For an appropriate choice of holomorphic function $f$ on $U$, the smooth $(n,0)$-form $\eta:=\chi f\,dz_1\wedge\dots\wedge dz_n$, which is compactly supported in $U$ and holomorphic in a neighborhood of $x$, will have any prescribed jet at $p$. The $(n,1)$-form $\bar\partial\eta$ is compactly supported in $U$, and identically zero in a neighborhood of $p$, so that $|\dbar\eta|_\om e^{-\psi}\in L^2(U)$. Since $dd^c\p\ge\om$ on $U$,~\cite[Th\'eor\`eme 5.1]{Dem82} yields an $L^2_{\mathrm{loc}}$ $(n,0)$-form $u$ on $M$ such that $\dbar u=\dbar\eta$ and
\begin{equation}\label{equ:l2}
i^{n^2} \int_M u\wedge\bar u\,e^{-2\psi}\le\int_U|\dbar\eta|^2_\om e^{-2\psi}dV_\om. 
\end{equation}
As a result, $v:=\eta-u$ is a holomorphic $n$-form on $X$. Since $u=\eta-v$ is holomorphic at $x$ and $\psi$ has an isolated singularity of type $(n+1)\log|z|$ at $x$, (\ref{equ:l2}) forces $u$ to vanish to order $2$ at $p$, so that $v$ and $\eta$ have the same $1$-jet at $p$. Finally, (\ref{equ:l2}) and the fact that $\psi$ is bounded above on $M$ shows that $u$ is $L^2$. Since $\eta$ is clearly $L^2$ as well, $v$ belongs to the Bergman space $\cH$, with given jet at $p$, and we are done. 
\end{proof}

%
%

Suppose now that $X$ is a compact complex manifold, and $\tilde X\to X$ is a Galois \'etale cover. Assume that the Bergman metric of $\tilde X$ is non-degenerate, so that the canonical metric $h_{\tilde X}$ on $K_{\tilde X}$ defined by $\cH(\tilde X)$ is smooth, strictly psh. Being invariant under automorphisms, this metric descends to a smooth, strictly psh metric on $K_X$, and the latter is thus ample by \cite{Kod}. This argument, which goes back to the same paper by Kodaira, admits the following variant. 

\begin{lem}\label{lem:big} Let $X$ be a compact K\"ahler manifold admitting a Galois \'etale cover $\tilde X\to X$ with generically non-degenerate Bergman metric. Then $X$ is projective and of general type. 
\end{lem}

\begin{proof} The assumption now means that the psh metric $h_{\tilde X}$ on $K_{\tilde X}$ is smooth and strictly psh on a non-empty Zariski open subset. It descends again to a psh metric on $K_X$, smooth and strictly psh on a non-empty Zariski open subset, and we conclude that $K_X$ is big by~\cite[\S 2.3]{Bou02} (see also~\cite[\S 1.5]{BEGZ10}). Being both Moishezon and K\"ahler, $X$ is then projective. 
\end{proof}

\begin{rmk}

Combining Lemma \ref{lem:1jet} with Lemma \ref{lem:big} one thus obtains that a compact K\"ahler manifold $X$ admitting a Galois \'etale cover $\tilde X$ which supports a bounded psh function that is strictly psh at some point is projective and of general type. 
It is perhaps worth noticing that a proof of this statement can also be obtained replacing the use of the first-named author's criterion in Lemma \ref{lem:big} by more elementary arguments using Poincar\'e series, following Gromov \cite[Corollary 3.2.B]{Gro} (see also \cite[Chapter 13]{Kol} for a very nice account of the method of Gromov). 
Indeed, the fact that the Galois cover $\tilde X$ posseses a bounded psh function that is strictly psh at some point implies both that $X$ has generically large fundamental group and that $K_{\tilde X}$ has a non-zero holomorphic $L^2$ section, the latter thanks to Lemma \ref{lem:1jet}. At this point one can use directly \cite[13.10 Corollary]{Kol}.
\end{rmk}

%
%
\section{Proof of Theorems A and B}
%
%

We finally prove the theorems stated in the introduction.

\subsection{Proof of Theorem A}

Let $X$ be a compact K\"ahler manifold with an \'etale cover $\pi:\tilde X\to X$ of bounded type, which may be assumed to be Galois after replacing $\tilde X$ by the universal cover of $X$. Since $\tilde X$ is also complete K\"ahler, its Bergman metric is non-degenerate by Lemma~\ref{lem:1jet}, and $X$ is thus projective and canonically polarized by \cite{Kod}. 

Now let $Y\subseteq X$ be an irreducible subvariety. On the one hand, pick any connected component $\tilde Y$ of the preimage $\pi^{-1}(Y)\subset\tilde X$, so that $\pi$ induces a Galois \'etale cover $\pi|_{\tilde Y}\colon\tilde Y\to Y$. On the other hand, let $\mu\colon Z\to Y$ be a projective modification with $Z$ smooth and $\mu$ isomorphic over $Y_{\reg}$, whose existence is guaranteed by Hironaka. Since $Y$ is K\"ahler and $\mu$ is projective, $Z$ is then a compact K\"ahler manifold. The fiber product $\tilde Z=Z\times_{Y }\tilde Y$ sits in the following diagram

\begin{center}
\begin{tikzcd}[row sep=small]
\tilde{Z}\ar[dr, "\tilde{\mu}"]\ar[dd,"\nu"'] & & \\
& \tilde{Y} \ar[r, hookrightarrow]\ar[dd, "\pi|_{\tilde{Y}}"] & \tilde{X} \ar[dd, "\pi"] \\
Z\ar[dr, "\mu"'] &  &  \\
 & Y \ar[r, hookrightarrow] & X
\end{tikzcd}
\end{center}
Being a base change of a Galois \'etale cover,  $\nu$ is a Galois \'etale cover, and $\tilde\mu$ is a resolution of singularities of $\tilde Y$. Since $\pi$ is \'etale, we have $\tilde Y_{\reg}=\pi^{-1}(Y_{\reg})$, and $\tilde\mu$ is an isomorphism over $\tilde Y_{\reg}$. The pull-back of $\f$ to $\tilde Z$ is thus a bounded psh function, strictly psh at any point $p\in\tilde\mu^{-1}(\tilde Y_{\reg})$. Since $Z$ is compact K\"ahler, $\tilde Z$ is complete K\"ahler. By Lemma~\ref{lem:1jet}, the Bergman metric of $Z$ is generically non-degenerate, and $Z$ is thus of general type, by Lemma~\ref{lem:big}. 

\subsection{Proof of Theorem B}

The first point of the statement follows directly from Lemma~\ref{lem:1jet} combined with Lemma \ref{lem:big}. 

Next, let $\tilde\Sigma\subset\tilde X$ be the set of points at which no bounded psh function on $\tilde X$ is strictly psh. This set is closed in the Euclidean topology and clearly invariant under holomorphic automorphisms of $\tilde X$. In particular, its image $\Sigma\subset X$ under the covering map is a proper closed subset.

For point (ii), the proof proceeds exactly as for the proof of point (iv) of Theorem A above. Indeed, we only used the following: given a subvariety $Y\subseteq X$ there exists a point on the regular locus of $Y$ whose preimage in the universal cover of $X$ contains a point of strict plurisubharmonicity of a bounded psh function on $\tilde X$. In our more general situation here, this happens precisely when $Y$ is not entirely contained in $\Sigma$.

Finally, for point (iii), by \cite[Theorem 3]{Sib81} we have that $\tilde X$ is hyperbolic at every point outside $\tilde \Sigma$. By this we mean that the infinitesimal Kobayashi metric of $\tilde X$ is locally uniformly bounded below away from zero at every such point. Therefore, the Kobayashi (pseudo)distance ---which is the integrated form of the infinitesimal Kobayashi metric--- of any two distinct points of $\tilde X$ must be positive unless they both belong to $\tilde \Sigma$, which is precisely the definition of $\tilde X$ being hyperbolic modulo $\tilde\Sigma$. Since $\tilde\Sigma$ is invariant and closed, by \cite[Theorem~3.2.32]{Kob98} we get that $X$ is hyperbolic modulo $\Sigma$. 

In particular, if $f\colon\mathbb C\to X$ is any holomorphic map, then any two distinct points in the image of $f$ have zero Kobayashi distance and must therefore sit inside $\Sigma$.

%
%
\section{Examples}\label{sect:examples}
%
%
The goal of this section is to briefly discuss examples of manifolds of bounded type with a compact free quotient. As we shall see, the list is unfortunately quite short. 

Consider first a bounded domain $\Om\Subset\C^n$ that admits a compact free quotient. 

\begin{itemize}
\item By~\cite{Sie50}, $\Om$ is automatically pseudoconvex (see also ~\cite[Theorem 6.2]{Kob59}).  

\item If $\partial\Om$ is $C^2$, then $\Om$ is biholomorphic to a ball~\cite{Won77,Ros79} (this already holds when $\partial\Om$ is $C^{1,1}$, by a recent result of Zimmer~\cite{Zim19}). 

\item If $\Om$ is homogeneous, then the Lie group $\Aut(\Om)$ admits a (uniform) lattice, hence is unimodular, and $\Om$ is thus necessarily symmetric~\cite[Theorem IV]{Han57}. 

\item If $\Om$ is convex hyperbolic, or if $\Om$ is irreducible and $\Aut(\Om)$ is positive dimensional, then $\Om$ is symmetric~\cite{Fra89,Fra95}. 

\end{itemize}

\subsubsection*{Bounded symmetric domains}

As mentioned in the introduction, a classical result of A.~Borel conversely implies that any bounded symmetric domain admits a compact free quotient~\cite{Bor63}. 

But even in this extensively studied framework, as far as we know, the first proof of the fact that compact free quotients of bounded symmetric domains satisfy Lang's conjecture (which corresponds to part (iv) of our Theorem A) is quite recent \cite{BKT13}. The proof by Brunebarbe--Klingler--Totaro is anyway of a different nature and appears to be perhaps less direct than ours, concerning mostly the negativity properties of the (holomorphic sectional and bisectional) curvature of the Bergman metric, which are under control in the symmetric case only. 

In this spirit, the same result also follows from Guenancia's work~\cite{Gue18}, since bounded symmetric domains enter the (\textsl{a priori} much larger) class (N) mentioned in the introduction (observe that while in \cite{BKT13} both the holomorphic bisectional and sectional curvatures have to be under control, \cite{Gue18} only requires a negative bound from above for the holomorphic sectional curvature).

\subsubsection*{Universal covers of Kodaira fibrations}

Examples of non-symmetric (and hence non-homogeneous) bounded domains with a compact free quotient arise from universal covers of Kodaira fibrations.

Recall that the latter are non-isotrivial holomorphic submersions with connected fibers $f\colon S\to C$ from a smooth projective surface to a smooth curve. Both the fiber and the base are then of genus at least $2$, examples arising for instance from the (projective) Satake compactification of the moduli space of curves of genus $g\ge 3$ (see \cite[\S V.14]{BHPV04}). 

By~\cite{Gri71}, the universal cover $\tilde S$ is biholomorphic to a bounded domain $\Om\Subset\C^2$, which is not symmetric. 

To see this, it is enough to check that $c_1^2(S)\ne 3 c_2(S)$ and $c_1(S)^2\ne 2c_2(S)$, by the Hirzebruch proportionality principle (since the only bounded symmetric domains in $\mathbb C^2$ are the ball and the bi-disc). While it is straightforward to see that the topological index $\tau(S)=\dfrac{1}{3}\bigl(c_1^2(S)-2c_2(S)\bigr)$ of a Kodaira fibration is always strictly positive, $c_1(S)^2\ne 3 c_2(S)$ is more involved~\cite{Liu96}.

Observe that for this family of examples, verifying Lang's conjecture is however straightforward since Kodaira fibrations are classically known to be of general type and, being Kobayashi hyperbolic surfaces, they cannot contain any (possibly singular) rational or elliptic curve.





\end{document}